\documentclass[11pt]{amsart}  % Specifies the document style.
\usepackage{amsthm, amsmath, amscd, amssymb, latexsym, stmaryrd, color}%txfonts 

\usepackage[all]{xypic}
\input xypic

\usepackage[top=3.0cm, bottom=3.0cm, left=3.0cm, right=3.0cm]{geometry}

\theoremstyle{plain}
\newtheorem{thm}{Theorem}
\newtheorem{theorem}{Theorem}[section]
\newtheorem{lemma}[theorem]{Lemma}
\newtheorem{proposition}[theorem]{Proposition}
\newtheorem{prop-def}[theorem]{Proposition-Definition}

\theoremstyle{definition}

\theoremstyle{remark}

\newtheorem*{ack}{Acknowledgement}

\usepackage[mathscr]{eucal}
\usepackage{graphics, graphpap}
\usepackage{array, tabularx, longtable}
\usepackage{color}
\numberwithin{equation}{section}

\def\Var{\mathrm{Var}}
\def\Mor{\mathrm{Mor}}
\def\Cons{\mathrm{Cons}}
\def\SA{\mathrm{SA}}
\def\loc{\mathrm{loc}}

\def\vol{\mathrm{vol}}
\def\pr{\mathrm{pr}}
\def\prsf{\mathsf{pr}}

\def\sr{\mathrm{sr}}

\def\ord{\mathrm{ord}}

\def\Spec{\mathrm{Spec}}
\def\Field{\mathrm{Field}}
\def\Id{\mathrm{Id}}
\def\acl{\mathrm{acl}}
\def\Ob{\mathrm{Ob}}

\def\DP{\mathrm{DP}}
\def\P{\mathrm{P}}
\def\PR{\mathrm{PR}}
\def\Rings{\mathrm{Rings}}

\def\Def{\mathrm{Def}}
\def\RDef{\mathrm{RDef}}
\def\GDef{\mathrm{GDef}}
\def\SDef{\mathrm{SDef}}
\def\ac{\mathrm{ac}}
\def\I{\mathrm{I}}

\def\L{\mathbb{L}}

\title[Grothendieck rings and measurable subassignments]{Grothendieck rings of definable subassignments\\ and equivariant motivic measures}  % Declares the document's title.
\author{L\^e Quy Thuong}
\dedicatory{\it Dedicated to Professors H\`a Huy Vui and Ta L\^e Loi on the special occasion of their birthdays}
\date{}
\address{University of Science, Vietnam National University, Hanoi \newline
\indent 334 Nguyen Trai street, Thanh Xuan district, Hanoi, Vietnam}
\email{leqthuong@gmail.com}

\thanks{The first author's research is funded by the Vietnam National University, Hanoi (VNU) under project number QG.19.06.}

\keywords{motivic measure, definable subassignment, measurable subassignment}
\subjclass[2010]{Primary 03C10, 14E18, 14G10}

\begin{document}           % End of preamble and beginning of text.

\begin{abstract}
The paper studies categories of definable subassignments with some category equivalences to semi-algebraic and constructible subsets of arc spaces of algebraic varieties. These materials allow us to compare the motivic measure of Cluckers-Loeser and the one of Denef-Loeser in certain classes of definable subassignments.
\end{abstract}

\maketitle  
               % Produces the title.

\section{Introduction}\label{sec1}
Since it was invented by Kontsevich at the 1995 Orsay seminar, geometric motivic integration has attained a full development and has become one of the central objects of algebraic geometry. From algebraic varieties to formal schemes, the development records contributions of several authors, such as Denef-Loeser \cite{DL1, DL2}, Sebag \cite{Se2004}, Loeser-Sebag \cite{LS}, Nicaise-Sebag \cite{NS, Ni2008}, Nicaise \cite{Ni2008}. Another point of view on motivic integration known as arithmetic motivic integration was developed almost at the same time, which works over only $p$-adic fields (see \cite{DL3}). Cluckers-Loeser's motivic integration \cite{CL2005, CL, CL2010}, which was built on model theory with respect to the Denef-Pas languages, is a general theory of motivic integration. It allows to specialize to both of arithmetic and geometric points of view (see \cite{GY2009}, \cite{CGH2014}, \cite{CR2019}). This theory of motivic integration has an important application to the fundamental lemma (see \cite{CHL}).
 
Recently, also using model theory with different languages, Hrushovski-Kazhdan \cite{HK} and Hrushovski-Loeser \cite{HL2016} have extended geometric motivic integration to the arithmetic aspect, with many interesting results and applications. 

Throughout the paper, the ground field $k$ will be always a field of characteristic zero. The present paper discusses the motivically measurable subassignments in the formalism of Cluckers-Loeser for motivic integration \cite{CL2005, CL, CL2010}, it also provides a comparison of their measure with the classical motivic measure of Denef-Loeser \cite{DL1, DL2}. By this purpose, we concentrate on a special Denef-Pas language $\mathcal L_{\DP,\P}$ with the Presburger language for value group sort, and consider the theory $T_{\acl}$ of algebraically closed fields containing $k$. Let $\Field_k$ be the category of all fields $K$ containing $k$, and $\Field_k(T_{\acl})$ the category of fields $K$ over $k$ such that each $(K(\!(t)\!), K, \mathbb Z)$ is a model of $T_{\acl}$. For $K$ in $\Field_k$, we consider the natural valuation map $\ord_t: K(\!(t)\!)^{\times}\to \mathbb Z$ augmented by $\ord_t(0)=+\infty$, and the natural angular component map $\overline{\ac}: K(\!(t)\!)^{\times}\to K$ augmented by  $\overline{\ac}(0)=0$. A basic affine subassignment $h[m,n,r]$ (or, in another notation, $h_{\mathbb A_{k(\!(t)\!)}^m\times \mathbb A_k^n\times \mathbb Z^r}$) is a functor $K\mapsto K(\!(t)\!)^m\times K^n\times\mathbb Z^r$ from $\Field_k$ to the category of sets. A definable subsassignment of $h[m,n,r]$ is a set of points in $h[m,n,r]$ satisfying a given formula $\varphi$; it is not a functor in general. In the first half of this paper, we study the categories concerning definable $T_{\acl}$-subsassignments $\SDef_k(\mathcal L_{\DP,\P}(k),T_{\acl})$ and $\SDef_k(X,\mathcal L_{\DP,\P}(k),T_{\acl})$, in which the language $\mathcal L_{\DP,\P}(k)$ is an extension of $\mathcal L_{\DP,\P}$ by adding constants in $k$ so that all polynomials in both valued field sort and residue field sort having coefficients in $k$.

The category $\SDef_k(\mathcal L_{\DP,\P}(k),T_{\acl})$ has objects to be small definable $T_{\acl}$-subassignments, which is comparable with the category $\SA_k$ of semi-algebraic subsets of the arc space of an algebraic $k$-variety. When fixing a $k$-variety $X$ we get the subcategory $\SDef_k(X,\mathcal L_{\DP,\P}(k),T_{\acl})$ of $\SDef_k(\mathcal L_{\DP,\P}(k))$ whose objects are all the small definable $T_{\acl}$-subassignments of $h_{X\times_k\Spec k(\!(t)\!)}$. The definition of $\SA_k$ and $\SA_k(X)$ is given in Section \ref{sect31}. The first main result of this paper is as follows.

\begin{thm}[Theorem \ref{equivalence}]
The categories $\SDef_k(\mathcal L_{\DP,\P}(k),T_{\acl})$ and $\SA_k$ are equivalent. If $X$ is an algebraic $k$-variety, then the categories $\SDef_k(X,\mathcal L_{\DP,\P}(k),T_{\acl})$ and $\SA_k(X)$ are equivalent.
\end{thm}

Let $S$ be an affine $k$-variety, and let $\RDef_{h_S}(\mathcal L_{\DP,\P}(k),T_{\acl})$ be the category whose objects $\mathsf X\to h_S$ are the $h_S$-projection of definable $T_{\acl}$-subassignment $\mathsf X$ of $h_S \times h_{\mathbb A_k^n}=h_{S\times_k \mathbb A_k^n}$ for some $n$ in $\mathbb N$. In Section \ref{Prel} we describe the category $\Cons_S$ of constructible morphisms from constructible sets over $k$ to $S$ in which a morphism in $\Cons_S$ from $X\to S$ to $Y\to S$ is determined uniquely up to $S$-isomorphism on $X$ by the graph of an $S$-morphism $X\to Y$. This category $\Cons_S$ has in fact the same Grothendieck ring with the category $\Var_S$. The second main result of the paper is stated as follows.

\begin{thm}[Theorem \ref{Lem41}]
For any $k$-variety $S$, the categories $\RDef_{h_S}(\mathcal L_{\DP,\P}(k),T_{\acl})$ and $\Cons_S$ are equivalent.
\end{thm}

This theorem has several interesting corollaries, such as the following isomorphism between Grothendieck rings $K_0(\RDef_{h_S}(\mathcal L_{\DP,\P}(k),T_{\acl}))\cong K_0(\Var_S)$. Thanks to this isomorphism we can identify the class $\L$ of the trivial line bundle $S\times_k\mathbb A_k^1\to S$ with the class $\big[h_{S\times_k\mathbb A_k^1}\to h_S\big]$. Moreover, if we put 
$$\mathbb A:=\mathbb Z\left[\L,\L^{-1},\frac{1}{1-\L^{-n}}\mid n\in \mathbb N^*\right],$$
we have $K_0(\RDef_k(\mathcal L_{\DP,\P}(k),T_{\acl}))\otimes_{\mathbb Z[\L]}\mathbb A\cong \mathscr M_{\loc}$. We also obtain the monodromic version $K_0^{\hat\mu}(\RDef_k(\mathcal L_{\DP,\P}(k),T_{\acl}))\cong K_0^{\hat\mu}(\Var_k)$ and $K_0^{\hat\mu}(\RDef_k(\mathcal L_{\DP,\P}(k),T_{\acl}))\otimes_{\mathbb Z[\L]}\mathbb A \cong \mathscr M_{\loc}^{\hat\mu}$. Here, $\mathscr M_{\loc}$ (resp. $\mathscr M_{\loc}^G$, with $G$ is either the group scheme $\mu_e$ or $\hat\mu=\varprojlim_e \mu_e$) is the localization of $K_0(\Var_k)$ (resp. $K_0^G(\Var_k)$) by inverting $\L$ and $\L^n-1$ for all $n$ in $\mathbb N^*$.

Theorem 10.1.1 of \cite{CL} implies that there is a unique functor from the category of definable subassignments to the category of abelian groups, $\mathsf X \mapsto \I C(\mathsf X)$, which assigns to $\mathsf X \to h_{\Spec k}$ a group morphism 
$$\mu: \I C(\mathsf X)\to \mathscr M_{\loc}$$ 
satisfying the axioms A0-A8 in that theorem. By \cite[Proposition 12.2.2]{CL}, if a definable subassignment $\mathsf X$ of $h[m,n,0]$ is bounded (see Section \ref{integrable}), the characteristic function $\mathbf 1_{\mathsf X}$ will be in $\I C(\mathsf X)$. In this case, $\mu(\mathsf X):=\mu(\mathbf 1_{\mathsf X})$ in $\mathscr M_{\loc}$ is the motivic measure of $\mathsf X$. When $\mathsf X$ is an invariant positively bounded definable subassignment of $h[m,n,0]$ we obtain the following comparison theorem (which also contains main results of the present paper).

\begin{thm}[Theorem \ref{comparison}, Proposition \ref{prop56}]
Let $\mathsf X$ be an invariant definable subassignment of $h[m,n,0]$ such that, for every $(x,y)$ on $\mathsf X$ with $x=(x_1,\dots,x_m)$, $\ord_t x_i\geq 0$ for all $1\leq i\leq m$. With the notions the morphism $\loc$ defined in Section \ref{Prel}, $\vol$ in Lemma \ref{defloc} and $\mathsf X[e]$ in the paragraph before Proposition \ref{prop55}, for $e\in \mathbb N^*$, the following identities hold:
\begin{align*}
\mu(\mathsf X)&=\loc(\vol(\mathsf X)) \qquad  \text{in}\ \ \mathscr M_{\loc},\\
\mu(\mathsf X[e])&=\loc(\vol(\mathsf X[e])) \quad  \text{in}\ \ \mathscr M_{\loc}^{\mu_e}.
\end{align*} 
\end{thm}
In fact, for $\mathsf X$ small in $h[m,0,0]$, the paper \cite{CL} showed early $\delta(\mu(\mathsf X))=\mu'(X)$ in $\widehat{\mathscr M}_k$, where $\widehat{\mathscr M}_k$ is a completion of $\mathscr M_k$ defined in \cite{DL2}, $\delta$ is the canonical morphism $\mathscr M_{\loc}\to \widehat{\mathscr M}_k$, $\mu'$ is the Denef-Loeser motivic volume defined in \cite{DL2}, and $X$ is the semi-algebraic subset of $\mathscr L(\mathbb A_k^m)$ corresponding to $\mathsf X$ via the equivalence of categories between $\SA_k(\mathbb A_k^m)$ and $\SDef_k(\mathbb A_k^m,\mathcal L_{\DP,\P}(k),T_{\acl})$ in Theorem \ref{equivalence}. 

In the end of this paper, we give a proof of the rationality of the series $\sum_{e\in \mathbb N^*}\mu(\mathsf X[e])T^e$ in $\mathscr M_{\loc}^{\hat\mu}[[T]]$ with $\mathsf X$ an invariant positively definable subassignment of $h[m,n,0]$.

\section{Grothendieck rings of varieties}\label{Prel}
%\subsection{Definitions}
Let $k$ be a field of characteristic zero, and $S$ an algebraic $k$-variety. As usual (cf. \cite{DL1, DL2}), we denote by $\Var_S$ the category of $S$-varieties and $K_0(\Var_S)$ its Grothendieck ring. By definition, $K_0(\Var_S)$ is the quotient of the free abelian group generated by the $S$-isomorphism classes $[X\to S]$ in $\Var_S$ modulo the following relation 
$$[X\to S]=[Y\to S]+[X\setminus Y\to S]$$ 
for $Y$ being Zariski closed in $X$. Together with fiber product over $S$, $K_0(\Var_X)$ is a commutative ring with unity $1=[\Id: S\to S]$. Put 
$$\L=[\mathbb A_k^1\times_k S\to S]$$ 
and write $\mathscr M_S$ for the localization of $K_0(\Var_S)$ inverting $\L$. Denote by $\mathscr M_{S,\loc}$ the localization of $\mathscr M_S$ inverting $\L^n-1$ for all $n$ in $\mathbb N^*$. %When $S=\Spec(k)$, we write simply $\mathscr M_k$ and $\mathscr M_{\loc}$ instead of $\mathscr M_{\Spec(k)}$ and $\mathscr M_{\Spec(k),\loc}$, respectively.

Let $C\Var_S$ be the category whose objects are constructible morphisms from constructible sets over $k$ to $S$ with the set of morphisms between objects $X\to S$ and $Y\to S$ given by
$$\Mor_{C\Var_S}(X\to S,Y\to S):=K_0(\Var_{X\times_SY}).$$
In other words, a morphism from $X\to S$ to $Y\to S$ in $C\Var_S$ is a finite sum of elements of the form $[U\to X\times_SY]$ (in $K_0(\Var_{X\times_SY})$), with $U$ being an algebraic $k$-variety. The composition of two basic morphisms $[U\to X\times_SY]$ and $[V\to Y\times_SZ]$ is the following morphism
$$[V\to Y\times_SZ]\circ [U\to X\times_SY]:=[U\times_YV\to X\times_SZ].$$
This definition makes sense since the morphism $U\times_YV\to X\times_SZ$ commutes with the structural morphisms to $S$, and it can be also extended by additivity. Clearly, the identity morphism of $X$ in $C\Var_k$ is the class in $K_0(\Var_{X\times_kX}$ of the diagonal morphism $X\to X\times_kX$.

Denote by $\Cons_S$ the subcategory of $C\Var_S$ in which objects of $\Cons_S$ are objects of $C\Var_S$ and a morphism of $\Cons_S$ from $X\to S$ to $Y\to S$ is an element 
$$[(\Id_X,f):X\to X\times_S Y]$$ 
in $K_0(\Var_{X\times_SY})$. By definition, each morphism of $\Cons_S$ from $X\to S$ to $Y\to S$ is determined uniquely, up to $S$-automorphism on $X$, by the constructible $S$-morphism of constructible sets $f:X\to Y$, or alternatively, by the graph of such an $f$. Using the above definition of Grothendieck ring for the category $\Cons_S$ we get 
$$K_0(\Var_S)\cong K_0(\Cons_S).$$

Let $X$ be an algebraic $k$-variety, and let $G$ be an algebraic group which acts on $X$. The $G$-action is called {\it good} if every $G$-orbit is contained in an affine open subset of $X$. Now we fix a good action of $G$ on the $k$-variety $S$. By definition, the $G$-equivariant Grothendieck group $K_0^G(\Var_S)$ of $G$-equivariant morphisms of $k$-varieties $X\to S$, where $X$ is endowed with a good $G$-action, is the quotient of the free abelian group generated by the $G$-equivariant isomorphism classes $[X\to S,\sigma]$ modulo the following relations
$$[X\to S,\sigma]=[Y\to S,\sigma|_Y]+[X\setminus Y\to S,\sigma|_{X\setminus Y}]$$
for $Y$ being $\sigma$-stable Zariski closed in $X$, and
$$[X\times_k\mathbb A_k^n\to S,\sigma]=[X\times_k\mathbb A_k^n\to S,\sigma']$$
if $\sigma$ and $\sigma'$ lift the same $G$-action on $X$ to an affine action on $X\times\mathbb A_k^n$. As above, we have the commutative ring with unity structure on $K_0^G(\Var_S)$ by fiber product, where $G$-action on the fiber product is through the diagonal $G$-action, and we may define the localization $\mathscr M_S^G$ of the ring $K_0^G(\Var_S)$ by inverting $\L$. In this article, we also consider the localization $\mathscr M_{S,\loc}^G$ of $\mathscr M_S^G$ with respect to the multiplicative family generated by the elements $1-\L^{-n}$ with $n$ in $\mathbb N^*$. 

Since a constructible subset $X$ of a $k$-variety is a finite disjoint union of locally closed subsets, we can endow $X$ with good $G$-action via its locally closed subsets. A constructible morphism is $G$-equivariant if its graph admits a good $G$-action induced from the ones on its source and target. So we can define categories $C\Var_S^G$ and $\Cons_S^G$ as follows. As avove, fix a good $G$-action on the $k$-variety $S$. Objects of $C\Var_S^G$ are  $G$-equivariant constructible morphisms from constructible sets endowed with a good $G$-action to $S$ (over $k$), the set of morphisms between objects $X\to S$ and $Y\to S$ is 
$$\Mor_{C\Var_S^G}(X\to S,Y\to S):=K_0(\Var_{X\times_SY}^G).$$ 
Objects of $\Cons_S^G$ are objects of $C\Var_S^G$, and a morphism of $\Cons_S^G$ from $X\to S$ to $Y\to S$ is 
$$\left[(\Id,f):X\to X\times_SY\right],$$ 
where $f:X\to Y$ is a $G$-equivariant constructible $S$-morphism of constructible sets. Similarly as previous, we can define the $G$-equivariant Grothendieck ring $K_0(\Cons_S^G)$ in the usual way, and obtain a canonical isomorphism of rings 
$$K_0^G(\Var_S)\cong K_0(\Cons_S^G).$$

Let $\hat\mu$ be the group scheme of roots of unity, which is the projective limit of group schemes $\mu_n=\Spec k[t]/(t^n-1)$ together transitions $\mu_{mn}\to \mu_n$ induced by $\lambda\mapsto \lambda^m$. A good $\hat{\mu}$-action on an $S$-variety $X$ is a good $\mu_n$-action on the $S$-variety $X$ for some $n$ in $\mathbb N^*$. We define 
$$K_0^{\hat \mu}(\Var_S)=\varinjlim  K_0^{\mu_n}(\Var_S), \quad \mathscr M_S^{\hat \mu}=K_0^{\hat \mu}(\Var_S)\left[\L^{-1}\right],$$ 
and 
$$\mathscr M_{S,\loc}^{\hat \mu}=K_0^{\hat \mu}(\Var_S)\left[\L^{-1},(\L^n-1)^{-1}\right]_{n\in \mathbb N^*}.$$
Clearly, we have the identities 
$$\mathscr M_S^{\hat \mu}=\varinjlim  \mathscr M_S^{\mu_n}\quad \text{and}\quad \mathscr M_{S,\loc}^{\hat \mu}=\varinjlim  \mathscr M_{S,\loc}^{\mu_n}.$$ 
By abuse of notation, we shall write $\loc$ for any of the following localization morphisms $\mathscr M_S \to \mathscr M_{S,\loc}$, $\mathscr M_S^{\mu_n} \to \mathscr M_{S,\loc}^{\mu_n}$ and $\mathscr M_S^{\hat\mu} \to \mathscr M_{S,\loc}^{\hat\mu}$.

When $S$ is $\Spec k$, we shall write simply $\Var_k$, $\mathscr M_k$, $\mathscr M_k^{G}$, $\mathscr M_{\loc}$ and $\mathscr M_{\loc}^G$ instead of $\Var_{\Spec k}$, $\mathscr M_{\Spec k}$, $\mathscr M_{\Spec k}^{G}$, $\mathscr M_{\Spec k,\loc}$ and $\mathscr M_{\Spec k,\loc}^G$, respectively.

\section{Arc spaces and rational series}
\subsection{Arc spaces}\label{sect31}
Let $X$ be an algebraic $k$-variety. For $e\in\mathbb N^*$, let $\mathscr L_e(X)$ be the space of $e$-jet schemes of $X$, which is actually a $k$-scheme representing the functor sending a $k$-algebra $A$ to the set of morphisms of $k$-schemes 
$$\Spec(A[t]/(t^{e+1}))\to X.$$
Thus, the set of $A$-rational points of $\mathscr L_e(X)$ is naturally identified with the set of $A[t]/(t^{e+1})$-rational points of $X$.

For $d\geq e$ in $\mathbb N^*$, the truncation modulo $t^{e+1}$ induces an affine morphism of $k$-schemes 
$$\mathscr L_d(X)\to \mathscr L_e(X)$$
denoted by $\pi_{e}^d$. If $X$ is a smooth variety of dimension $d$, the morphism $\pi_{e}^d$ is a locally trivial fibration with fiber $\mathbb{A}_k^{(d-e)\dim_kX}$.

The above jet schemes $\mathscr L_e(X)$ and truncation morphisms $\pi_{e}^d$ form in a natural way a projective system of $k$-schemes. As the truncation morphisms are affine, the projective limit of this system exists in the category of $k$-schemes and is called \emph{arc space of $X$} and denoted by $\mathscr L(X)$ with truncation morphisms 
$$\pi_e :\mathscr{L}(X) \rightarrow \mathscr L_e(X).$$ 
If $k\subseteq K$ is a field extension of $k$, then the 
$K$-rational points of $\mathscr{L}(X)$ correspond one-to-one to the $K[\![t]\!]$-rational points of $X$.  

Recall from \cite[Section 2]{DL2}, for any algebraically closed field $K$ containing $k$, that a subset of $K(\!(t)\!)^m\times \mathbb Z^r$ is {\it semi-algebraic} if it is a finite boolean combination of sets of the forms
\begin{align}\label{eq31}
\left\{(x,\alpha)\in K(\!(t)\!)^m\times \mathbb Z^r \mid \ord_tf(x)\geq \ord_tg(x)+\ell(\alpha)\right\},
\end{align}
and
\begin{align}\label{eq32}
\left\{(x,\alpha)\in K(\!(t)\!)^m\times \mathbb Z^r \mid\ord_tf(x)\equiv \ell(\alpha)\mod n\right\},
\end{align}
and
\begin{align}\label{eq33}
\left\{(x,\alpha)\in K(\!(t)\!)^m\times \mathbb Z^r \mid\Phi(\overline{\ac}(f_1(x)),\dots,\overline{\ac}(f_p(x)))=0\right\},
\end{align}
where $f$, $g$, $f_i$ and $\Phi$ are $k$-polynomials, $\ell$ is a $\mathbb Z$-polynomial of degree at most 1, $n$ is in $\mathbb N$, and $\overline{\ac}(f_i(x))$ is the angular component of $f_i(x)$. One calls a collection of formulas defining a semi-algebraic set a {\it semi-algebraic condition}. A subset $A$ of $\mathscr L(X)$ is called {\it semi-algebraic} if there exists a covering of $X$ by affine Zariski open sets $U$ such that $A\cap \mathscr L(U)$ is of the form
\begin{align}\label{eq3.02}
A\cap \mathscr L(U)=\{x\in \mathscr L(U) \mid \theta(f_1(\tilde{x}),\dots,f_p(\tilde{x});\alpha)\},
\end{align}
where $f_i$ are regular functions on $U$, $\theta$ is a semi-algebraic condition, $\alpha$ may be a given tuple of integers or nothing, and $\tilde{x}$ is the element in $\mathscr L(U)(k(x))$ corresponding to a point $x$ in $\mathscr L(U)$ of residue field $k(x)$. 

By \cite{Pas}, if $g: X\to Y$ is a morphism of algebraic $k$-varieties and $A$ is a semi-algebraic subset of $\mathscr L(X)$, then $g(A)$ is a semi-algebraic subset of $Y$. Then the map $g: A\to g(A)$ is called a {\it semi-algebraic morphism} of semi-algebraic sets. More generally, let $A$ and $B$ be semi-algebraic subsets of $\mathscr L(X)$ and $\mathscr L(Y)$, the arc spaces of $k$-varieties $X$ and $Y$, respectively, and let $h: A\to B$ be a map. Then $h$ is called a {\it semi-algebraic morphism} if its graph is a semi-algebraic subset of $\mathscr L(X\times_kY)$. Denote by $\SA_k$ be the category whose objects are pairs $(A,\mathscr L(X))$, where $A$ is a semi-algebraic subset of the arc space $\mathscr L(X)$ of an algebraic $k$-variety $X$, and a morphism of $\SA_k$ between two objects $(A,\mathscr L(X))$ and $(B,\mathscr L(Y))$ is a semi-algebraic morphism of semi-algebraic sets $A\to B$. For a given $k$-variety $X$, we can consider the full subcategory $\SA_k(X)$ of $\SA_k$ consists of semi-algebraic subsets of $\mathscr L(X)$.

In the sense of \cite[Definition-Proposition 3.2]{DL2}, Denef-Loeser's motivic volume is defined on $\Ob\SA_k(X)$ the set of all the semi-algebraic subsets of $\mathscr L(X)$ with the reasonable properties. By \cite[Remark 16.3.2]{CL}, this motivic volume essentially takes values in $\mathscr M_{\loc}$. In the present article, we denote Denef-Loeser's motivic volume by $\mu'$ (the symbol $\mu$ will be devoted to indicate Cluckers-Loeser's motivic volume \cite{CL}).

\subsection{Rationality}
Let $\mathscr M$ be a commutative ring with unity containing $\L$ and $\L^{-1}$, and let $\mathscr M[[T]]$ be the set of formal power series in $T$ with coefficients in $\mathscr M$, which is a ring and also a $\mathscr M$-module with respect to usual operations for series. Denote by $\mathscr M[[T]]_{\sr}$ the submodule of $\mathscr M[[T]]$ generated by 1 and by finite products of terms 
$$\frac{\L^pT^q}{(1-\L^pT^q)}$$ 
for $(p,q)$ in $\mathbb{Z}\times\mathbb{N}^*$. An element of $\mathscr M[[T]]_{\sr}$ is called a {\it rational} series. By \cite{DL1}, there exists a unique $\mathscr M$-linear morphism 
$$\lim_{T\to\infty}: \mathscr M[[T]]_{\sr}\to \mathscr M$$ 
such that for any $(p,q)$ in $\mathbb{Z}\times\mathbb{N}^*$,
$$\lim_{T\to\infty}\frac{\L^pT^q}{(1-\L^pT^q)}=-1.$$ 
 
Let us recall some examples on the rationality. Let $X$ be a smooth algebraic $k$-variety of pure dimension $m$, and $f$ a regular function on $X$ with zero locus $X_0\not=\emptyset$. For $e$ in $\mathbb N^*$, put 
$$X[e]=\left\{\gamma \in \mathscr{L}_e(X)\mid f(\gamma)= t^e\mod t^{e+1}\right\},$$
which is naturally an $X_0$-variety and stable under the action $\lambda\cdot\gamma(t):=\gamma(\lambda t)$ of $\mu_e$ on $\mathscr L_e(X)$. 
Write simply $\big[X[e]\big]$ for the class $\big[X[e]\to X_0\big]$ in $\mathscr M_{X_0}^{\mu_e}$. It is proved in \cite{DL1} that the series
\begin{align*}
Z_f(T):=\sum_{e\in \mathbb N^*}\big[X[e]\big]\L^{-em}T^e, 
\end{align*}
is a rational series, i.e., in $\mathscr M_{X_0}^{\hat{\mu}}[[T]]_{\sr}$. More general, we can obtain the rationality of a series generalizing $Z_f(T)$ without assuming that $X$ is smooth, and with $f$ concerning several semi-algebraic subsets in $\mathscr L(X)$. Let $\mu'$ be Denef-Loeser's motivic volume defined in \cite{DL2}. The below theorem is a result given in \cite[Proposition 4.6]{LN2020}.

\begin{theorem}[L\^e-Nguyen \cite{LN2020}]\label{Thm34}
Let $X$ be a $k$-variety and $f$ a regular function on $X$. Let $A_{\alpha}$, $\alpha$ in $\mathbb N^r$, be a family of semi-algebraic subsets of $\mathscr L(X)$ such that there exists a covering of $X$ by affine Zariski open sets $U$ satisfying the condition that $A_{\alpha}\cap \mathscr L(U)$ are finite boolean combinations of sets of the forms (\ref{eq31}) and (\ref{eq32}). Assume that, for every $\alpha$ in $\mathbb N^r$, $A_{\alpha}$ is stable in the sense of \cite{DL2} and disjoint with $\mathscr L(X_{\text{Sing}})$. For $e\in \mathbb N^*$, we put 
$$A_{e,\alpha}:=\left\{\gamma\in A_{\alpha} \mid f(\gamma)=t^e\mod t^{e+1}\right\}.$$
Let $\Delta$ be a rational polyhedral convex cone in $\mathbb{R}^{r+1}_{\geq 0}$ and $\bar{\Delta}$ its closure. Let $\ell$ and $\ell'$ be integral linear forms on $\mathbb Z^{r+1}$ with $\ell(e,\alpha)>0$ and $\ell'(e,\alpha)\geq 0$ for all $(e,\alpha)$ in $\bar{\Delta}\setminus \{0\}$. Then the formal power series
$$Z(T):=\sum_{(e,\alpha)\in \Delta\cap \mathbb N^{r+1}}\mu'\left(A_{e,\alpha}\right)\L^{-\ell'(e,\alpha)}T^{\ell(e,\alpha)}$$
is an element of $\mathscr M_k^{\hat{\mu}}[[T]]_{\sr}$, and the limit $\lim_{T\to \infty}Z(T)$ is independent of such an $\ell$ and $\ell'$.
\end{theorem}

%*******************************
\section{Categories of definable subassignments}\label{SS4}
In this section, we recall some concepts and results on motivic integration in the sense of Cluckers-Loeser \cite{CL}. We also provide an equivariant version for objects concerning definable $T_{\acl}$-subassignments, where $T_{\acl}$ is the theory of all algebraically closed fields containing $k$.

\subsection{Definable subassignments}\label{Sec4.1.1}
We consider the formalism of Cluckers and Loeser \cite{CL} with a concrete Denef-Pas language $\mathcal L_{\DP,\P}$ consisting of the ring language $\mathbf L_{\Rings}=\left\{+,-,\cdot,0,1\right\}$ for valued fields, also the ring language $\mathbf L_{\Rings}$ for residue fields, and the Presburger language $\mathbf L_{\PR}$ for value groups, where 
$$\mathbf L_{\PR}=\left\{+,-,0,1,\leq\right\}\cup \left\{\equiv_n\mid n\in \mathbb N^*\right\},$$ 
and $\equiv_n$ is the equivalence relation modulo $n$.

Let $\Field_k$ be the category of all fields $K$ containing $k$ whose morphisms are field morphisms. For any $K$ in $\Field_k$, we consider the natural valuation map $\ord_t: K(\!(t)\!)^{\times}\to \mathbb Z$ augmented by $\ord_t(0)=+\infty$, and the natural angular component map $\overline{\ac}: K(\!(t)\!)^{\times}\to K$ augmented by  $\overline{\ac}(0)=0$.

For a basic set 
$$V:=\mathbb A_{k(\!(t)\!)}^m\times \mathbb A_k^n\times \mathbb Z^r,$$ 
with $m$, $n$, $r$ in $\mathbb N$, we consider the functor $h_V$ (also denoted by $h[m,n,r]$) from $\Field_k$ to the category of sets defined by 
$$h_V(K)=h[m,n,r](K):=K(\!(t)\!)^m\times K^n\times\mathbb Z^r.$$ 
If $\mathsf X$ is a map sending each object $K$ of $\Field_k$ to a subset $\mathsf X(K)$ of $K(\!(t)\!)^m\times K^n\times\mathbb Z^r$, then $\mathsf X$ is called a {\it affine subassignment} (or, shortly speaking, {\it subassignment}) of $h_V=h[m,n,r]$. Note that $\mathsf X$ is not necessarily a subfunctor of $h[m,n,r]$. In the same way, we can define morphisms of subassignments and their graph, as well as union, subtraction, Cartesian product and fiber product of two subassignments. 

A subassignment $\mathsf X$ of $h[m,n,r]$ is called {\it definable} if there exists a formula $\varphi$ in $\mathcal L_{\DP,\P}$ with $k(\!(t)\!)$-coefficients and $m$ free variables in the valued field sort, $k$-coefficients and $n$ free variables in residue field sort, and $r$ free variables in the value group sort, such that, for any $K$ in $\Field_k$,
$$\mathsf X(K)=\left\{x\in K(\!(t)\!)^m\times K^n\times\mathbb Z^r \mid (K(\!(t)\!),K,\mathbb Z) \models \varphi(x)\right\}.$$
In this setting, we also write $h_{\varphi}$ for the definable subassignment $\mathsf X$. Denote by $\emptyset$ the empty definable subassignment, with $\emptyset(K)=\emptyset$ for any $K$ in $\Field_k$. For $\mathsf X$ and $\mathsf X'$ being definable subassignments of $h[m,n,r]$ and $h[m',n',r']$, respectively, a {\it definable} morphism $\mathsf X\to \mathsf X'$ is a morphism of subassignments $\mathsf X\to \mathsf X'$ such that its graph is a definable subassignment of $h[m+m',n+n',r+r']$. 

For a set 
$$W:=\mathcal X\times X\times\mathbb Z^r,$$ 
with $\mathcal X$ an algebraic $k(\!(t)\!)$-variety, and $X$ an algebraic $k$-variety, we define 
$$h_{W}(K):=\mathcal X(K(\!(t)\!))\times X(K)\times\mathbb Z^r,$$ 
for any $K$ in $\Field_k$. In general, we can define definable subassignments of $h_{W}$, definable morphisms of definable subassignments, and usual operations on definable subassignments of functors of the form $h_{W}$ using a glueing procedure (as in \cite[Section 2.3]{CL}): taking finite covers (always exist) of $\mathcal X$ and $X$ by affine open $k(\!(t)\!)$-subvarieties and $k$-subvarieties, respectively, going back to the definition of affine definable subassignment, and glueing them.  

%\subsubsection{}
We consider the category $\Def_k(\mathcal L_{\DP,\P})$ (or $\Def_k$ for short) of affine definable subassignments, where its objects are pairs $(\mathsf X,h[m,n,r])$ with $\mathsf X$ being a definable subassignment of $h[m,n,r]$, and a morphism 
$$(\mathsf X,h[m,n,r])\to (\mathsf X',h[m',n',r'])$$ 
in $\Def_k$ is a definable morphism $\mathsf X\to \mathsf X'$. We also consider the category $\GDef_k(\mathcal L_{\DP,\P})$ (or $\GDef_k$ for short) of global definable subassignments, where objects of $\GDef_k(\mathcal L_{\DP,\P})$ are pairs $(\mathsf X,h_{W})$ with $h_{W}$ as above and $\mathsf X$ being a definable subassignment of $h_{W}$, and a morphism 
$$(\mathsf X,h_W)\to (\mathsf X',h_{W'})$$ 
in $\GDef_k$ is a definable morphism $\mathsf X\to \mathsf X'$. For any affine definable subassignment $\mathsf S$, we denote by $\Def_{\mathsf S}(\mathcal L_{\DP,\P})$ (or $\Def_{\mathsf S}$ for short) the category of morphisms $\mathsf X\to \mathsf S$ in $\Def_k$, and a morphism in $\Def_{\mathsf S}$ between $\mathsf X\to \mathsf S$ and $\mathsf X'\to \mathsf S$ is a morphism $\mathsf X\to \mathsf X'$ in $\Def_k$ which is compatible with the morphisms to $\mathsf S$. For any definable subassignment $\mathsf S$, the category $\GDef_{\mathsf S}(\mathcal L_{\DP,\P})$ (or $\GDef_{\mathsf S}$ for short) can be defined in the same way as $\GDef_k$ with $h_{\Spec k}$ replaced by $\mathsf S$. 

When we consider the category $\Field_k(T_{\acl})$ of all algebraically closed fields containing $k$ in stead of $\Field_k$, and use the language $\mathcal L_{\DP,\P}(k)$ (see the definition in Section \ref{categories}) instead of $\mathcal L_{\DP,\P}$, we obtain corresponding notions of $T_{\acl}$-subassignment, $\Def_k(\mathcal L_{\DP,\P}(k),T_{\acl})$ and  $\GDef_k(\mathcal L_{\DP,\P}(k),T_{\acl})$.

\subsection{Points on definable subassignments}
Let $\mathsf X$ be an object in $\GDef_k$. A point $x$ on $\mathsf X$ is a tuple $x=(x_0,K)$ such that $K$ is in $\Field_k$ and $x_0$ is in $\mathsf X(K)$. For such a point $x$ on $\mathsf X$ we usually write $k(x)$ for $K$ and call it the residue field of $x$. Let 
$$f: \mathsf X \to \mathsf Y$$ 
be a morphism in $\Def_k$, with 
$$\mathsf X=(\mathsf X_0,h[m,n,r])$$ 
and 
$$\mathsf Y=(\mathsf Y_0,h[m',n',r']),$$ 
whose graph is defined by a formula $\varphi(x,y)$, where $x$ is in $h[m,n,r]$ and $y$ is in $h[m',n',r']$. One defines the fiber of $f$ over a point $y=(y_0,k(y))$ on $\mathsf Y$ to be the definable subsassignment $\mathsf X_y$ in $\Def_{k(y)}$ given by the formula $\varphi(x,y_0)$. In the category $\GDef_k$, fibers of a morphism are defined in the same way by using affine covers.

\subsection{Categories $\SDef_k(\mathcal L_{\DP,\P}(k),T_{\acl})$, $\SDef_k(X,\mathcal L_{\DP,\P}(k),T_{\acl})$, $\RDef_{h_S}(\mathcal L_{\DP,\P}(k),T_{\acl})$}\label{categories} \quad\\
Denote by $\mathcal L_{\DP,\P}(k)$ the language extending $\mathcal L_{\DP,\P}$ by adding constants in $k$ so that all polynomials in both valued field sort and residue field sort having coefficients in $k$. Let $X$ be an algebraic $k$-variety, let 
$$\mathcal X:=X\times_k\Spec k(\!(t)\!),$$ 
and $\mathsf A$ a definable subassignment of $h_{\mathcal X}$ defined by a formula in $\mathcal L_{\DP,\P}(k)$. Assume first $\mathcal X$ is a closed subscheme in $\mathbb A_{k(\!(t)\!)}^m$, for some $m$ in $\mathbb N$, such that the ideal defining $\mathcal X$ generated by polynomials with coefficients in $k[[t]]$. The above mentioned definable subassignment $\mathsf A$ (i.e., defined by a formula in $\mathcal L_{\DP,\P}(k)$) is called {\it small} if $\mathsf A$ is contained in the following definable subassignment 
$$\left\{(x_1,\dots,x_m)\in h[m,0,0] \mid \ord_tx_i\geq 0, 1\leq i\leq m\right\}.$$
For $\mathcal X$ not necessarily affine, we call $\mathsf A$ {\it small} if there exists a cover of $\mathcal X$ by open affine $k(\!(t)\!)$-subvarieties $\mathcal U_i$ defined by the vanishing of polynomials with coefficients in $k[[t]]$ such that $\mathsf A\cap h_{\mathcal U_i}$ are small for all $i$. Let $\SDef_k(\mathcal L_{\DP,\P}(k),T_{\acl})$ be the subcategory of $\GDef_k(\mathcal L_{\DP,\P},T_{\acl})$ whose objects are pairs 
$$(\mathsf A,h_{X\times_k\Spec k(\!(t)\!)}),$$ 
where $X$ is an algebraic $k$-variety and $\mathsf A$ is a small definable $T_{\acl}$-subassignment of $h_{X\times_k\Spec k(\!(t)\!)}$, and a morphism in $\SDef_k(\mathcal L_{\DP,\P}(k),T_{\acl})$ between objects 
$$(\mathsf A,h_{X\times_k\Spec k(\!(t)\!)})$$ 
and 
$$(\mathsf B,h_{Y\times_k\Spec k(\!(t)\!)})$$ 
is a $T_{\acl}$-morphism of $T_{\acl}$-subassignments $\mathsf A\to \mathsf B$ such that its graph is a small definable $T_{\acl}$-subassignment of $h_{X\times_kY\times_k\Spec k(\!(t)\!)}$.

Fixing an algebraic $k$-variety $X$, we define a category denoted by $\SDef_k(X,\mathcal L_{\DP,\P}(k),T_{\acl})$ which is the full subcategory of $\SDef_k(\mathcal L_{\DP,\P}(k),T_{\acl})$ whose objects contain all the small definable $T_{\acl}$-subassignments of $h_{X\times_k\Spec k(\!(t)\!)}$. 

\begin{theorem}\label{equivalence}
The categories $\SDef_k(\mathcal L_{\DP,\P}(k),T_{\acl})$ and $\SA_k$ are equivalent. If $X$ is an algebraic $k$-variety, then the categories $\SDef_k(X,\mathcal L_{\DP,\P}(k),T_{\acl})$ and $\SA_k(X)$ are equivalent.
\end{theorem}

\begin{proof}
At the moment, for short, we write $(\mathsf A,X)$ instead of $(\mathsf A,h_{X\times_k\Spec k(\!(t)\!)})$ for an object of $\SDef_k(\mathcal L_{\DP,\P}(k),T_{\acl})$, and $(A,X)$ instead of $(A,\mathscr L(X))$ for an object in of $\SA_k$. 

First, let us construct a functor $\mathcal F$ from $\SDef_k(\mathcal L_{\DP,\P}(k),T_{\acl})$ to $\SA_k$. Let $(\mathsf A,X)$ be an object of $\SDef_k(X,\mathcal L_{\DP,\P}(k),T_{\acl})$. Then, there is a cover of $X\times_k\Spec k(\!(t)\!)$ by Zariski open affine $k(\!(t)\!)$-subvarieties $\mathcal U$, with $\mathcal U$ embedded as a closed $k(\!(t)\!)$-subvariety in some $\mathbb A_{k(\!(t)\!)}^m$ (we can take $m$ common for all $\mathcal U$) and the embedding defined over $k[[t]]$, such that for the standard coordinates $x_i$ of $h[m,0,0]$ and any point $x$ on $\mathsf A\cap h_{\mathcal U}$ we have $ \ord_tx_i(x)\geq 0$. %We can choose $x_i$ such that $x_i$ are standard coordinate components of $h[m,0,0]$ for $1\leq i\leq n$, and that $h_{\mathcal U}$ is defined by the vanishing of $x_i$ for $n+1\leq i\leq m$. 
Moreover, $\mathsf A\cap h_{\mathcal U}$ is defined by a formula $\varphi(\underline{x},\alpha)$ in the language $\mathcal L_{\DP,\P}(k)$, where $\underline{x}=(x_1,\dots,x_m)$ and $\alpha=(\alpha_1,\dots,\alpha_r)$ are free variables in value group sort, namely,
$$\mathsf A\cap h_{\mathcal U}=\left\{x\in h_{\mathcal U} \mid \varphi(x_1(x),\dots,x_m(x), \alpha)\right\}.$$
By Denef-Pas' quantifier elimination for algebraically closed fields (cf. Corollary 2.1.2 of \cite{CL}), $\varphi(\underline{x},\alpha)$ is equivalent to a finite disjunction of formulas of the form 
\begin{align}\label{fsmall}
\psi(\overline{\ac}g_1(\underline{x}),\dots,\overline{\ac}g_q(\underline{x}))\wedge \vartheta(\ord_t f_1(\underline{x}),\dots,\ord_t f_p(\underline{x}),\alpha),
\end{align}
where $f_i$ and $g_j$ are polynomials over $k$, $\psi$ is an $\mathbf L_{\Rings}$-formula with coefficients in $k$, and $\vartheta$ is an $\mathbf L_{\PR}$-formula. Thus, as seen in (\ref{eq31}), (\ref{eq32}) and (\ref{eq33}), the formula $\varphi$ is nothing but a semi-algebraic condition. 

Since all polynomials $f_i$ and $g_j$ have coefficients in $k$, in particular, polynomials defining $\mathcal U$ have coefficients in $k$, there exists a unique closed $k$-subvariety $U$ in $\mathbb A_k^m$ such that
\begin{align}\label{existU}
\mathcal U=U\times_k\Spec k(\!(t)\!),
\end{align}  
hence we have a cover $\{U\}_U$ of $X$ by Zariski open affine $k$-subvarieties. The above $x_i$ induce regular functions $x'_i$ on $U$ such that $x'_i$ are standard coordinate components in $\mathbb A_k^m$ for every $1\leq i\leq n$, and that $U$ is defined by the vanishing of $x'_i$ for $n+1\leq i\leq m$. Now we put 
$$A_U:=\left\{x\in \mathscr L(U)\mid \varphi(x'_1(\tilde{x}),\dots,x'_m(\tilde{x}),\alpha)\right\},$$
where $\tilde{x}$ is defined after (\ref{eq3.02}), and glue $A_U$'s into a semi-algebraic subset $A$ of $\mathscr L(X)$. Note that the construction of $A$ is up to semi-algebraic isomorphism independent of the choice of the cover $\{\mathcal U\}_{\mathcal U}$. Let us define 
$$\mathcal F(\mathsf A,X):=(A,X),$$ 
which is an object in $\SA_k$.

We shall construct a morphism $\mathcal F(\mathsf f)$ in $\SA_k$ which is the image under $\mathcal F$ of a morphism $\mathsf f$ in $\SDef_k(\mathcal L_{\DP,\P}(k),T_{\acl})$ such that 
\begin{align}\label{preservingg}
\mathcal F(\mathsf g\circ \mathsf f)=\mathcal F(\mathsf g)\circ \mathcal F(\mathsf f).
\end{align}
Consider a morphism $\mathsf f: (\mathsf A,X) \to (\mathsf B,Y)$, written $\mathsf f: \mathsf A \to \mathsf B$ for short, in $\SDef_k(\mathcal L_{\DP,\P}(k),T_{\acl})$. Then we can cover $X\times_kk(\!(t)\!)$ (resp. $Y\times_kk(\!(t)\!))$ by Zariski open affine $k(\!(t)\!)$-subvarieties $\mathcal U$ (resp. $\mathcal U'$) such that $\mathsf A\cap h_{\mathcal U}$ and $\mathsf B\cap h_{\mathcal U'}$ are defined by formulas, say $\psi_{\mathcal U}$ and $\phi_{\mathcal U'}$, respectively, which are disjunction of formulas of the form (\ref{fsmall}). We use the covering $\{\mathcal U\}$ of $X\times_kk(\!(t)\!)$ and the formulas $\psi_{\mathcal U}$ (resp. the covering $\{\mathcal U'\}$ of $Y\times_kk(\!(t)\!)$ and the formulas $\phi_{\mathcal U'}$) to construct a constructible subset $\mathcal A$ (resp. $\mathcal B$) of $X\times_kk(\!(t)\!)$ (resp. $Y\times_kk(\!(t)\!)$). The morphism $\mathsf f$ induces a morphism of constructible sets $\mathcal A\to \mathcal B$, hence with the same reason as the existence of $U$ in (\ref{existU}) the morphism $\mathcal A\to \mathcal B$ in its turn induces a semi-algebraic morphism of semi-algebraic sets $f: (A,X) \to (B,Y)$. So we define 
$$\mathcal F(\mathsf f):=f,$$
which is well defined and a morphism in $\SA_k$.

Since any morphism $\mathsf f$ can be factorized through the inclusion into its graph followed by a projection, to check the preserving property (\ref{preservingg}) it suffices to do for inclusions and projections of small definable $T_{\acl}$-subassignments. By definition, for the $\mathsf B$-projection 
$$\prsf_{\mathsf B}: (\mathsf A \times \mathsf B,X\times_k Y) \to (\mathsf B,Y)$$ 
we have 
$$\mathcal F(\prsf_{\mathsf B})=\pr_B: A\times B \to B,$$ 
the $B$-projection of semi-algebraic subset $A\times B$ of $\mathscr L(X\times_kY)=\mathscr L(X)\times_k\mathscr L(Y)$ onto $B$, which is a morphism in $\SA_k$. Also, for a morphism $\mathsf i_{\mathsf{AB}}: (\mathsf A,X)\to (\mathsf B,X)$ in $\SDef_k(\mathcal L_{\DP,\P}(k),T_{\acl})$ induced by an inclusion $\mathsf A \hookrightarrow \mathsf B$ of small definable $T_{\acl}$-subassignments, we have $\mathcal F(\mathsf i_{\mathsf{AB}})(x)=x$ for all $x\in A$. Now, let us consider  
$$\mathsf i: (\mathsf C,X\times_k Y) \to (\mathsf A\times \mathsf B, X\times_k Y),$$ 
which is a morphism in the category $\SDef_k(\mathcal L_{\DP,\P}(k),T_{\acl})$ induced by an inclusion $\mathsf C \hookrightarrow \mathsf A\times \mathsf B$. By definition, it is clear that  
$$\mathcal F(\prsf_{\mathsf B}\circ \mathsf i)=\mathcal F(\prsf_{\mathsf B})\circ \mathcal F(\mathsf i).$$ 
In the same way, for a morphism
$$\mathsf j: (\mathsf B,Y) \to (\mathsf E, Y)$$ 
in $\SDef_k(\mathcal L_{\DP,\P}(k))$ induced by an inclusion $\mathsf B \hookrightarrow \mathsf E$, we have 
$$\mathcal F(\mathsf j \circ \prsf_{\mathsf B})=\mathcal F(\mathsf j)\circ \mathcal F(\prsf_{\mathsf B}).$$

We now construct a functor $\mathcal G$ from $\SA_k$ to $\SDef_k(\mathcal L_{\DP,\P}(k),T_{\acl})$ which is naturally inverse to $\mathcal F$. Let $(A,X)$ be an object of $\SA_k$, i.e., $A$ is a semi-algebraic subset of $\mathscr L(X)$. By definition, there exist a cover of $X$ by Zariski open affine $k$-subvarieties $V$ (viewed as a closed $k$-subvariety of $\mathbb A_k^m$), and for each $V$, regular functions $h_i$ on $V$, $1\leq i\leq m$, a semi-algebraic condition $\varphi$ (with $h_i$ and $\varphi$ depending on $V$) such that 
$$A\cap \mathscr L(V)=\left\{x\in \mathscr L(V) \mid \varphi(h_1(\tilde{x}),\dots,h_m(\tilde{x}),\alpha)\right\}.$$ 
Clearly, 
$$\mathcal V:=V\times_k\Spec k(\!(t)\!)$$ 
is embedded over $k[[t]]$ into $\mathbb A_{k(\!(t)\!)}^m$, and they form a cover of $X\times_k\Spec k(\!(t)\!)$. Note that $\varphi$ is a formula in the language $\mathcal L_{\DP,\P}(k)$, and that each $h_i$ induces a definable morphism of definable subassignments 
$$x_i: h_{\mathcal V}\to h[1,0,0].$$ 
Put
$$\mathsf A_{\mathcal V}=\left\{x\in h_{\mathcal V}\mid \varphi(x_1(x),\dots,x_m(x),\alpha), \ord_tx_i(x)\geq 0, 1\leq i\leq m\right\},$$
and glue all $\mathsf A_{\mathcal V}$ along the cover $\{\mathcal V\}$ of $X\times_k\Spec k(\!(t)\!)$ to get a small definable $T_{\acl}$-subassignment $\mathsf A$ of $h_{X\times_k\Spec k(\!(t)\!)}$. We can prove that the construction of $\mathsf A$ is up to definable isomorphism independent of the choice of the cover $\{V\}$. So we can define 
$$\mathcal G(A,X)=(\mathsf A,X),$$ 
which is an object of $\SDef_k(\mathcal L_{\DP,\P}(k),T_{\acl})$. Similarly, we can define $\mathcal G(f)$ to be a morphism of $\SDef_k(\mathcal L_{\DP,\P}(k),T_{\acl})$ when $f$ is a morphism of $\SA_k$, which satisfies 
$$\mathcal G(f\circ g)=\mathcal G(f) \circ \mathcal G(g).$$
 
The existence of natural isomorphisms
$$\varepsilon: \mathcal F \circ \mathcal G \to \Id_{\SA_k}$$ 
and 
$$\eta: \Id_{\SDef_k(\mathcal L_{\DP,\P}(k),T_{\acl})} \to \mathcal G \circ \mathcal F$$ 
follows from the fact that the construction of $\mathsf A$ from $A$ and vice versa is independent of the choice of covers by open affine subvarieties. 
\end{proof}

Let $\mathsf S$ be an object in $\GDef_k(\mathcal L_{\DP,\P}(k),T_{\acl})$. Let $\RDef_{\mathsf S}(\mathcal L_{\DP,\P}(k),T_{\acl})$ denote the full subcategory of $\GDef_{\mathsf S}(\mathcal L_{\DP,\P}(k),T_{\acl})$ such that each object $\mathsf X\to \mathsf S$ of $\RDef_{\mathsf S}(\mathcal L_{\DP,\P}(k),T_{\acl})$ is the $\mathsf S$-projection of a definable subassignment $\mathsf X$ of $\mathsf S \times h_{\mathbb A_k^n}$, for some $n$ in $\mathbb N$. We first mention a special case when $\mathsf S=h_S$ as follows. Let $S$ be a closed $k$-subvariety of $\mathbb A_k^d$, for a given $d$ in $\mathbb N$. Then the category $\RDef_{h_S}(\mathcal L_{\DP,\P}(k),T_{\acl})$ defined previously is just the full subcategory of $\Def_{h_S}(\mathcal L_{\DP,\P}(k),T_{\acl})$, its objects $\mathsf X\to h_S$ are the $h_S$-projection of definable subassignments $\mathsf X$ of $h_{S\times_k\mathbb A_k^n}$, with $n$ being variable in $\mathbb N$. 

\begin{theorem}\label{Lem41}
For any $k$-variety $S$, the categories $\RDef_{h_S}(\mathcal L_{\DP,\P}(k),T_{\acl})$ and $\Cons_S$ are equivalent. 
\end{theorem}

\begin{proof}
Using the argument in \cite[Section 16.2]{CL}, under some elimination theorems, we have that objects of $\Def_k(\mathcal L_{\DP,\P}(k),T_{\acl})$ are defined by formulas without quantifiers in the Denef-Pas language $\mathcal L_{\DP,\P}$. Since Chevalley's constructibility theorem in algebraic geometry (over $k$) is nothing else than the quantifier elimination theorem for the theory of algebraically closed fields containing $k$, a formula defining a definable subassignment of $h_{S\times_k\mathbb A_k^n}$ defines a constructible subset of $S\times_k\mathbb A_k^n$, and via graph, a definable morphism of definable subassignments gives rise to a constructible morphism of constructible sets, and vice versa. For a detailed argument, we can use the strategy in the proof of Theorem \ref{equivalence}.
\end{proof}

\subsection{Actions}
%Let us establish an equivariant version for $\RDef_{h_S}$. 
Let $X$ be an algebraic $k$-variety, and $G$ an algebraic group over $k$. A {\it $G$-action} (or {\it $h_G$-action}) on $h_X$ is a definable morphism of definable subassignments 
$$h_{G\times_k X}\to h_X$$ 
such that the corresponding morphism of $k$-varieties
$$G\times_k X\to X$$ 
is a $G$-action on $X$. The $G$-action on $h_X$ is called {\it good} if the corresponding $G$-action on $X$ is good. In this setting, a definable morphism of definable subassignments 
$$h_X\to h_Y$$ 
is {\it $G$-equivariant} if the corresponding morphism of $k$-varieties 
$$X\to Y$$ 
is $G$-equivariant. By Theorem \ref{Lem41}, we can extend this definition of good $G$-action to that on any definable subassignment of $h[0,n,0]$, for $n$ in $\mathbb N$. Let $S$ be a closed $k$-subvariety of $\mathbb A_k^d$, and let $S$ be endowed with a good $G$-action. Denote by $\RDef_{h_S}^G(\mathcal L_{\DP,\P}(k),T_{\acl})$ the subcategory of $\GDef_{h_S}(\mathcal L_{\DP,\P}(k),T_{\acl})$ whose objects are $G$-equivariant definable $T_{\acl}$-morphisms of definable $T_{\acl}$-subassignments $\mathsf X\to h_S$, where $\mathsf X$ is a definable $T_{\acl}$-subassignment of $h_{S\times_k\mathbb A_k^n}$, for some $n$ in $\mathbb N$, and $\mathsf X$ is endowed with a good $G$-action, a morphism in $\RDef_{h_S}^G(\mathcal L_{\DP,\P}(k),T_{\acl})$ from an object $\mathsf X\to h_S$ to another one $\mathsf Y\to h_S$ is a $G$-equivariant definable $T_{\acl}$-morphism $\mathsf X \to \mathsf Y$ which commutes with the $G$-equivariant morphisms to $h_S$. 

\begin{lemma}\label{Lem42}
For any $k$-variety $S$, $\RDef_{h_S}^G(\mathcal L_{\DP,\P}(k),T_{\acl})$ and $\Cons_S^G$ are equivalent.
\end{lemma}

\begin{proof}
The lemma is deduced directly from Theorem \ref{Lem41} and the definition of good $G$-action on definable subassignments.
\end{proof}

For an algebraic $k(\!(t)\!)$-variety $\mathcal X$, the definable subassignment $h_{\mathcal X}$ admits a natural $\mu_n$-action $h_{\mu_n}\times h_{\mathcal X} \to h_{\mathcal X}$ induced by 
$$(\lambda,t)\mapsto \lambda t,$$ 
for all $n$ in $\mathbb N^*$. More precisely, for every $K$ in $\Field_k$, $\lambda$ in $\mu_n(K)$ and $\varphi(t)$ in $\mathcal X(K(\!(t)\!))$, we have 
$$\lambda\cdot \varphi(t)=\varphi(\lambda t).$$ 
The profinite group scheme $\hat\mu$ acts naturally on $h_{\mathcal X}$ via $\mu_n$ for some $n$ in $\mathbb N^*$.

\subsection{Grothendieck semirings and rings of definable subassignments}
Let $\mathsf S$ be a definable subassignment. According to \cite{CL}, the Grothendieck semigroup $SK_0(\RDef_{\mathsf S})$ of the category $\RDef_{\mathsf S}$ is the quotient of the free abelian semigroup generated by symbols $[\mathsf X\to \mathsf S]$ with $\mathsf X\to \mathsf S$ being objects in $\RDef_{\mathsf S}$ modulo the following relations: 
$$[\emptyset\to \mathsf S]=0,$$ 
$$[\mathsf X\to \mathsf S]=[\mathsf Y\to \mathsf S]$$ 
if $\mathsf X\to \mathsf S$ and $\mathsf Y\to \mathsf S$ are isomorphic in $\RDef_{\mathsf S}$, and
$$[\mathsf X\cup \mathsf Y\to \mathsf S]+[\mathsf X\cap \mathsf Y\to \mathsf S]=[\mathsf X\to \mathsf S]+[\mathsf Y\to \mathsf S]$$ 
for definable subassignments $\mathsf X$ and $\mathsf Y$ of $\mathsf S\times h_{\mathbb A_k^n}$, for some $n$ in $\mathbb N$, and morphisms of $\mathsf X$ and $\mathsf Y$ to $\mathsf S$ factorizing through $\mathsf S$-projection. Denote by $K_0(\RDef_{\mathsf S})$ the group associated to the Grothendieck semigroup $SK_0(\RDef_{\mathsf S})$. If we provide $SK_0(\RDef_{\mathsf S})$ and $K_0(\RDef_{\mathsf S})$ with a product induced by the fiber product over $\mathsf S$ of morphisms of subassignments to $\mathsf S$ defined in Section 2.2 of \cite{CL}, then they are commutative semiring and ring with unity, respectively. Remark that the canonical morphism
$$SK_0(\RDef_{\mathsf S})\to K_0(\RDef_{\mathsf S})$$
is not necessarily injective.

Let $S$ be a $k$-variety endowed with a given $G$-action. 
The {\it $G$-equivariant Grothendieck group $K_0^G(\RDef_{h_S})$} is the quotient of the free abelian group generated by symbols 
$$[\mathsf X\to h_S,\sigma]$$ 
with $\mathsf X$ being a definable subassignment of $h_{S\times_k \mathbb A_k^n}$, for some $n$ in $\mathbb N$, endowed with a good $G$-action $\sigma$, and $\mathsf X \to h_S$ being a morphism in $\Def_k$, modulo the following relations
$$[\mathsf X \to h_S,\sigma]=[\mathsf Y\to h_S,\sigma']$$
if there exists a $G$-equivariant definable morphism $\mathsf X \to \mathsf Y$ which commutes with the definable morphisms to $h_S$,
$$[\mathsf X\to h_S,\sigma]=[\mathsf Y\to h_S,\sigma|_{\mathsf Y}]+[\mathsf X\setminus \mathsf Y\to h_S,\sigma|_{\mathsf X\setminus \mathsf Y}]$$
for $\mathsf Y$ being $\sigma$-stable definable subassignment of $\mathsf X$, and
$$[\mathsf X\times h_{\mathbb A_k^m}\to h_S,\sigma]=[\mathsf X\times h_{\mathbb A_k^m}\to h_S,\sigma']$$ 
if $\sigma$ and $\sigma'$ lift the same $G$-action on $\mathsf X$ to an affine action on $\mathsf X\times h_{\mathbb A_k^m}$, for any $m\geq 0$. As above, with respect to fiber product of subassignments endowed with diagonal $G$-action, the group $K_0^G(\RDef_{h_S})$ is a commutative rings with unity.

The Grothendieck rings $K_0(\RDef_{\mathsf S}(\mathcal L_{\DP,\P}(k),T_{\acl}))$ and $K_0(\RDef_{h_S}^G(\mathcal L_{\DP,\P}(k),T_{\acl}))$ of the categories $\RDef_{\mathsf S}(\mathcal L_{\DP,\P}(k),T_{\acl})$ and $\RDef_{h_S}^G(\mathcal L_{\DP,\P}(k),T_{\acl})$, respectively, are defined in the same way as previous.

\begin{lemma}
For any $k$-variety $S$, there are canonical isomorphisms
$$K_0(\RDef_{h_S}(\mathcal L_{\DP,\P}(k),T_{\acl}))\cong K_0(\Var_S)$$ 
and 
$$K_0(\RDef_{h_S}^G(\mathcal L_{\DP,\P}(k),T_{\acl}))\cong K_0^G(\Var_S).$$
\end{lemma}

\begin{proof}
This statement is a direct corollary of Theorem \ref{Lem41} and Lemma \ref{Lem42}. We can also refer to \cite[Section 16.2]{CL} for a proof of the first isomorphism.
%The first isomorphism was already shown in Section 16.2 of \cite{CL}. In order to prove that $K_0^{\hat\mu}(\RDef_{h_S})\cong K_0^{\hat\mu}(\Var_S)$, we take $\mathsf Y=h_Y$ to be an arbitrary definable subassignment of $h_{S\times_k\mathbb A_k^n}$, for some $n$ in $\mathbb N$, with $Y$ a $k$-subvariety of $S\times_k\mathbb A_k^n$. If $h_{\hat\mu}\times \mathsf Y\to \mathsf Y$ is a good $\hat\mu$-action on $\mathsf Y\to h_S$, there exists a unique $\hat\mu$-action $\hat\mu\times Y\to Y$ on $Y$ compatible with it via the canonical isomorphism $K_0(\RDef_{h_S})\cong K_0(\Var_S)$, hence we have $K_0^{\hat\mu}(\RDef_{h_S})\cong K_0^{\hat\mu}(\Var_S)$.
\end{proof}

\section{Integrable functions and measurable subassignments}\label{SS5}
\subsection{Rings of motivic functions and Functions}
Let $\mathsf S$ be a definable subassignment. Put 
$$\mathbb A:=\mathbb Z\left[\L,\L^{-1},\left(\frac{1}{1-\L^{-n}}\right)_{n\in \mathbb N^*}\right],$$
where by abuse of notation $\L$ also stands for the class of $\mathsf S\times h_{\mathbb A_k^1}$ in $K_0(\RDef_{\mathsf S})$. By \cite{CL}, for any real number $q>1$, there is a unique morphism of rings 
$$v_q: \mathbb A\to \mathbb R$$ 
sending $\L$ to $q$, and such that, whenever $q$ is transcendental, $v_q$ is injective. Denote by $\mathbb A_+$ the subset of $\mathbb A$ consisting of elements $a$ with $v_q(a)\geq 0$. 

We now recall Section 4.6 of \cite{CL}. Denote by $|\mathsf S|$ the set of points of $\mathsf S$. Let $\mathscr P(\mathsf S)$ be %the ring of constructible Presburger functions on $\mathsf S$, that is, 
the subring of the ring of functions $|\mathsf S|\to\mathbb A$ which is generated by constant functions 
$$|\mathsf S|\to\mathbb A,$$ 
by functions 
$$\tilde{\alpha}:|\mathsf S|\to\mathbb Z,$$ 
and by functions 
$$\L^{\tilde{\beta}}: |\mathsf S|\to \mathbb A,$$ 
for definable morphisms $\alpha, \beta: \mathsf S\to h_{\mathbb Z}=h[0,0,1]$. Here, notice that to any definable morphism 
$$\alpha: \mathsf S\to h[0,0,1]$$ 
corresponds a function 
$$\tilde{\alpha}:|\mathsf S|\to\mathbb Z.$$ 
Denote by $\mathscr P_+(\mathsf S)$ the semiring of functions in $\mathscr P(\mathsf S)$ with values in $\mathbb A_+$. In particular, the ring $\mathscr P(h_{\Spec k})$ and the semiring $\mathscr P_+(h_{\Spec k})$ are nothing but $\mathbb A$ and $\mathbb A_+$, respectively. Denote by $\mathscr P^0(\mathsf S)$ the subring of $\mathscr P(\mathsf S)$ which is generated by $\L-1$ and by character functions $\mathbf 1_{\mathsf X}$ for all definable subassignments $\mathsf X$ of $\mathsf S$, and also define 
$$\mathscr P_+^0(\mathsf S):=\mathscr P^0(\mathsf S)\cap \mathscr P_+(\mathsf S).$$

According to \cite[Section 5.3]{CL}, the semiring $\mathscr C_+(\mathsf S)$ of {\it positive constructible motivic functions} on $\mathsf S$ and the ring $\mathscr C(\mathsf S)$ of {\it constructible motivic functions} on $\mathsf S$ are defined as follows
\begin{equation}\label{eq5.1}
\begin{aligned}
\mathscr C_+(\mathsf S)&:=SK_0(\RDef_{\mathsf S})\otimes_{\mathscr P_+^0(\mathsf S)}\mathscr P_+(\mathsf S),\\
\mathscr C(\mathsf S)&:=K_0(\RDef_{\mathsf S})\otimes_{\mathscr P^0(\mathsf S)}\mathscr P(\mathsf S).
\end{aligned}
\end{equation}
If $S$ is an algebraic $k$-variety endowed with a good $\hat\mu$-action, we define 
$$\mathscr C^{\hat\mu}(h_S):=K_0^{\hat\mu}(\RDef_{h_S})\otimes_{\mathscr P^0(h_S)}\mathscr P(h_S).$$

As mentioned in Section 16.1 of \cite{CL}, with respect to the language $\mathcal L_{\DP,\P}(k)$ and theory $T_{\acl}$ we can define rings $\mathscr P_+(\mathsf S,(\mathcal L_{\DP,\P}(k),T_{\acl}))$, $\mathscr P(\mathsf S,(\mathcal L_{\DP,\P}(k),T_{\acl}))$, $\mathscr C_+(\mathsf S,(\mathcal L_{\DP,\P}(k),T_{\acl}))$, and in the same way, the ring $\mathscr C^{\hat\mu}(h_S,(\mathcal L_{\DP,\P}(k),T_{\acl}))$.

In the rest of the article we will not work with the rings $\mathscr C(h_S)$, $\mathscr C(h_S,(\mathcal L_{\DP,\P}(k),T_{\acl}))$ and $\mathscr C^{\hat\mu}(h_S,(\mathcal L_{\DP,\P}(k),T_{\acl}))$ except the trivial case $S=\Spec k$. For this trivial case, we have the following lemma, which is is obvious from the definition.

\begin{lemma}\label{prepare1}
There exist canonical isomorphisms 
\begin{align*}
\mathscr C_+(h_{\Spec k})&\cong SK_0(\RDef_k)\otimes_{\mathbb N[\L-1]}\mathbb A_+,\\
\mathscr C(h_{\Spec k},(\mathcal L_{\DP,\P}(k),T_{\acl}))& \cong \mathscr M_{\loc},\\
\mathscr C^{\hat\mu}(h_{\Spec k},(\mathcal L_{\DP,\P}(k),T_{\acl}))& \cong \mathscr M_{\loc}^{\hat\mu}.
\end{align*}
%$$\mathscr C_+(h_{\Spec k})\cong SK_0(\RDef_k)\otimes_{\mathbb N[\L-1]}\mathbb A_+,$$ 
%$$\mathscr C(h_{\Spec k},(\mathcal L_{\DP,\P}(k),T_{\acl})) \cong \mathscr M_{\loc}\quad \text{and}\quad \mathscr C^{\hat\mu}(h_{\Spec k},(\mathcal L_{\DP,\P}(k),T_{\acl})) \cong \mathscr M_{\loc}^{\hat\mu}.$$
\end{lemma}

The important properties of $\mathscr C_+(\mathsf S)$ and $\mathscr C(\mathsf S)$ are given in Section 5 of \cite{CL}.

According to \cite[Section 3]{CL}, the K-dimension of a definable subassignment (with K a capital letter not a mathematical notation) is defined as follows. If $\mathsf S$ is a definable subassignment of $h_{\mathcal X}$, with $\mathcal X$ being an algebraic $k(\!(t)\!)$-variety, then the K-{\it dimension} of $\mathsf S$, denoted by $\text{Kdim} \mathsf S$, is the dimension of the $k(\!(t)\!)$-variety which is the intersection of all $k(\!(t)\!)$-subvarieties $\mathcal Y$ of $\mathcal X$ with $h_{\mathcal Y}$ containing $\mathsf S$. It may happen that the intersection is empty; in that case, we define 
$$\text{Kdim}\mathsf S:=-\infty.$$ 
If $\mathsf S$ is a definable subassignment of $h_{\mathcal X\times X\times \mathbb Z^r}$, with $\mathcal X$ as above and $X$ being a $k$-variety, then we define
$$\text{Kdim}\mathsf S:=\text{Kdim}\prsf_1(\mathsf S),$$ 
where 
$$\prsf_1:h_{\mathcal X\times X\times \mathbb Z^r}\to h_{\mathcal X}$$ 
is the first projection. 

A positive constructible motivic function $\varphi$ in $\mathscr C_+(\mathsf S)$ is called of K-dimension $\leq d$ if $\varphi$ is a finite sum $\sum_i\alpha_i\mathbf 1_{\mathsf S_i}$ in $\mathscr C_+(\mathsf S)$ such that the K-dimension of every $\mathsf S_i$ is $\leq d$. Let $\mathscr C_+^{\leq d}(\mathsf S)$ be the sub-semigroup of $\mathscr C_+(\mathsf S)$ of elements of K-dimension $\leq d$, and  
$$C_+^d(\mathsf S):=\mathscr C_+^{\leq d}(\mathsf S)/\mathscr C_+^{\leq d-1}(\mathsf S)$$
and 
$$C_+(\mathsf S):=\bigoplus_{d\geq 0}C_+^d(\mathsf S).$$ 
An element in $C_+(\mathsf S)$ is called a {\it positive constructible motivic Function} on $\mathsf S$ (with the capital letter F). Clearly, $C_+(\mathsf S)$ is a graded abelian semigroup and has a module structure over the semiring $\mathscr C_+(\mathsf S)$ (cf. \cite[Section 6]{CL}).

%Defining a subgroup $\I_{\mathsf S}C(\mathsf X)$ of $\mathsf S$-integrable functions in $C(\mathsf X)$ is a main task in \cite{CL}, so that, firstly, under \cite[Theorems 10.1.1 and 14.1.1]{CL}, the existence and uniqueness of the semigroup $\I_{\mathsf S}C_+(\mathsf X)$ corresponding to $\I_{\mathsf S}C(\mathsf )$ are guaranteed by a list of eight axioms. In fact, $\I_{\mathsf S}C_+(\mathsf X)$ is a graded semingroup and contained in the semingroup $C_+(\mathsf X)$, where $C_+(\mathsf X)$ is defined in the same way as $C(\mathsf X)$ with, at starting point, $\mathbb A$ replaced by $\mathbb A_+$. By definition, a function $\varphi\in C(\mathsf X)$ is called {\it $\mathsf S$-integrable}, i.e., $\varphi\in \I_{\mathsf S}C(\mathsf X)$, if we may write $\varphi=\iota(\varphi_+)-\iota(\varphi_-)$, with $\varphi_+, \varphi_-\in \I_{\mathsf S}C_+(\mathsf X)$, and $\iota$ being the canonical morphism $C_+(\mathsf X)\to C(\mathsf X)$.

\subsection{Integrable positive Functions and measurable subassignments} \label{integrable}
Let $\mathsf S$ be in $\Def_k$. By \cite[Theorem 10.1.1]{CL}, there exists a unique functor $\I_{\mathsf S}C_+$ from the category $\Def_{\mathsf S}$ to the category of abelian semigroups which sends every morphism 
$$\mathsf f:\mathsf X \to \mathsf Y$$ 
in $\Def_{\mathsf S}$ to a morphism of semigroups 
$$\mathsf f_!: \I_{\mathsf S}C_+(\mathsf X) \to \I_{\mathsf S}C_+(\mathsf Y)$$ 
and satisfies the axioms A0--A8. If $\mathsf S=h_{\Spec k}$ we write $\I C_+(\mathsf X)$ instead of $\I_{\mathsf S}C_+(\mathsf X)$, and call it the semigroup of integrable positive Functions on $\mathsf X$. By Proposition 12.2.2 of \cite{CL}, if $\mathsf X$ is a definable subassignment of $h[m,n,0]$ which is {\it bounded}, i.e., there exists an $s\in \mathbb N$ such that $\mathsf X$ is contained in the subassignment of $h[m,n,0]$ defined by 
$$\ord_t x_i\geq -s$$ 
for all $1\leq i\leq m$, then $[\mathbf 1_{\mathsf X}]$ belongs to $\I C_+(\mathsf X)$, where $\mathbf 1_{\mathsf X}$ is the characteristic function on $\mathsf X$. (In the previous definition of boundedness, if $s=0$ then $\mathsf X$ is said to be {\it positively bounded}.) %If, in addtion, $\mathsf X=h_{\Spec k}$, we have 
%$$\I C_+(h_{\Spec k})\cong \mathscr C_+(h_{\Spec k})\cong SK_0(\RDef_k)\otimes_{\mathbb N[\L-1]}\mathbb A_+.$$

Also in the trivial case $\mathsf S=h_{\Spec k}$, let us take $\mathsf f$ to be the projection of $\mathsf X$ onto the final subassignment $h_{\Spec k}$ of $\Def_k$. We denote by $\widetilde\mu$ the morphism of semigroups $\mathsf f_!$, namely,
$$\widetilde \mu: \I C_+(\mathsf X)\to \I C_+(h_{\Spec k})\cong \mathscr C_+(h_{\Spec k}).$$
Applying Section 16.1 of \cite{CL} we have a canonical morphism of rings 
$$\mathscr C_+(h_{\Spec k})\to \mathscr C_+(h_{\Spec k},(\mathcal L_{\DP,\P}(k),T_{\acl})).$$
On the other hand, we also have another canonical morphism
$$\mathscr C_+(h_{\Spec k},(\mathcal L_{\DP,\P}(k),T_{\acl}))\to \mathscr C(h_{\Spec k},(\mathcal L_{\DP,\P}(k),T_{\acl}))\cong \mathscr M_{\loc}.$$
Taking the composition of the last two morphisms with $\widetilde\mu$ we get a morphism of rings
$$\mu: \I C_+(\mathsf X)\to \mathscr M_{\loc}.$$

As mentioned previously, if $\mathsf X$ is a bounded definable subassignment in $\Def_k$, then $[\mathbf 1_{\mathsf X}]$ is in $\I C_+(\mathsf X)$. In this case, we call $\mathsf X$ a {\it motivically measurable} (definable) subaasignment. We define the {\it motivic measure} of $\mathsf X$ to be 
$$\mu(\mathsf X):=\mu([\mathbf 1_{\mathsf X}]),$$ 
which lies in $\mathscr M_{\loc}$. By the additivity of the integral (Axiom A2 in \cite[Theorem 10.1.1]{CL}), the motivic measure $\mu$ is additive on bounded definable subassignments.

Denote by $\widehat{\mathscr M}_k$ the completion of $\mathscr M_k$ in the sense of \cite{DL2}, and by $\delta$ the canonical morphism $\mathscr M_{\loc}\to \widehat{\mathscr M}_k$ defined by the expansion of $1-\L^{-n}$, for every $n$ in $\mathbb N^*$.

\begin{proposition}
Let $X$ be an algebraic $k$-variety, $A$ a semi-algebraic subset of $\mathscr L(X)$, and $\mathsf A$ the small definable subassignment corresponding to $A$ via the equivalence of categories between $\SA_k(X)$ and $\SDef_k(X,(\mathcal L_{\DP,\P}(k),T_{\acl}))$ in Theorem \ref{equivalence}. Then 
$$\delta(\mu(\mathsf A))=\mu'(A),$$ 
where $\mu'$ is Denef-Loeser's motivic volume defined in \cite{DL2}. 
\end{proposition}

\begin{proof}
Note that if a definable function 
$$\alpha: \mathsf A\to h_{\mathbb Z}$$ 
in the language $\mathcal L_{\DP,\P}(k)$ is the zero function on $\mathsf A$, then the semi-algebraic function 
$$\tilde\alpha: A\to \mathbb Z$$ 
corresponding to $\alpha$ via the equivalence of categories of $\SA_k(X)$ and $\SDef_k(X,(\mathcal L_{\DP,\P}(k),T_{\acl}))$ in Theorem \ref{equivalence} is also the zero function on $A$. Now applying Theorem 16.3.1 of \cite{CL} to $\alpha=0$ we get the proposition.
\end{proof}

\subsection{Invariant definable subassignments and their measure}
Let $m, n$ be in $\mathbb N$, and $\gamma=(\gamma_1,\dots,\gamma_m)$ in $\mathbb Z^m$. A definable subassignment $\mathsf X$ of $h[m,n,0]$ is called {\it $\gamma$-invariant} if, for every $K$ in $\Field_k$, $(a,b)$ and $(x,y)$ in 
$$h[m,n,0](K)=K(\!(t)\!)^m\times K^n$$ 
satisfying 
$$\ord_tx_i\geq\gamma_i$$ 
for $1\leq i\leq m$, both elements $(a,b)$ and $(a,b)+(x,y)$ are simultaneously in either $\mathsf X(K)$ or in the complement of $\mathsf X(K)$ in $K(\!(t)\!)^m\times K^n$. A definable subassignment of $h[m,n,0]$ is called {\it invariant} if it is $\gamma$-invariant for some $\gamma$ in $\mathbb Z^m$. In the case $\gamma_i=\beta \in \mathbb Z$ for all $1\leq i\leq m$, we write $\beta$-invariant instead of $\gamma$-invariant. Note that if $\mathsf X$ is $\gamma$-invariant and $\gamma_i\leq \gamma'_i$ for all $1\leq i\leq m$, then $\mathsf X$ is also $(\gamma'_1,\dots,\gamma'_m)$-invariant. It is a fact that any bounded definable subassignment of $h[m,0,0]$ closed in the valuation topology is $\gamma$-invariant for some $\gamma$ in $\mathbb Z^m$. 

Now, let $\beta$ be in $\mathbb N^*$, and let $\mathsf X$ be a bounded definable subassignment of $h[m,n,0]$ with $\ord_t x_i\geq 0$ for every $x=(x_1,\dots,x_m)$ on $\mathsf X$ and for all $1\leq i\leq m$. Then $\mathsf X$ is $\beta$-invariant if and only if there exists a constructible subset $X_{\beta}$ of 
$$\mathbb A_k^{\beta m}\times_k \mathbb A_k^n \cong \mathscr L_{\beta-1}(\mathbb A_k^m) \times \mathbb A_k^n$$ 
such that, for every $K$ in $\Field_k$, 
$$\mathsf X(K)=\left(\pi_{\beta}(K)\right)^{-1}\left(X_{\beta}(K)\right),$$ 
which is the pullback of $X_{\beta}(K)$ under the canonical map 
$$\pi_{\beta}(K): K[[t]]^m \times K^n \to \big(K[t]/(t^{\beta})\big)^m\times K^n\cong K^{\beta m}\times K^n.$$ 
Indeed, we can show this by observing that, for $K$ in $\Field_k$, every fiber of the restriction $\pi_{\beta}(K)$ is definably bijective to the set $(t^{\beta})K[[t]]^m$. The maps $\pi_{\beta}(K)$ induce the canonical morphism 
$$\pi_{\beta}: h[m,n,0]\to h[0,\beta m+n,0],$$
and $\mathsf X$ is $\beta$-invariant if and only if there exists a constructible subset $X_{\beta}\subseteq \mathbb A_k^{\beta m}\times_k \mathbb A_k^n$ such that $\mathsf X=\pi_{\beta}^{-1}(h_{X_{\beta}})$. By abuse of notation, we shall denote the canonical morphism $\mathsf X\to h_{X_{\beta}}$ by $\pi_{\beta}$.

\begin{lemma}\label{defloc}
Let $\beta \leq \beta'$ be in $\mathbb N^*$, let $\mathsf X$ and $X_{\beta}$ be as previous. Then the identity 
$$[h_{X_{\beta'}}]=[h_{X_{\beta}}]\L^{(\beta'-\beta)m}$$ 
holds in $K_0(\RDef_k)$. As a consequence, the element $[h_{X_{\beta}}]\L^{-(\beta+1)m}$ in $K_0(\RDef_k)[\mathbb L^{-1}]$ is independent of the choice of sufficiently large $\beta$.
\end{lemma}

\begin{proof}
The natural map 
$$\mathscr L_{\beta'-1}(\mathbb A_k^m) \times_k \mathbb A_k^n \to \mathscr L_{\beta-1}(\mathbb A_k^m) \times_k \mathbb A_k^n$$ 
induced by truncation is a Zariski locally trivial fibration with fiber $\mathbb A_k^{(\beta'-\beta)m}$. Along this map, $X_{\beta'}$ is the preimage of $X_{\beta}$. Thus we get the identity $[h_{X_{\beta'}}]=[h_{X_{\beta}}]\L^{(\beta'-\beta)m}$ in $K_0(\RDef_k)$.
\end{proof}

We denote by $\vol(\mathsf X)$ the image of $[h_{X_{\beta}}]\L^{-(\beta+1)m}\in K_0(\RDef_k)[\mathbb L^{-1}]$ under the canonical morphism (due to \cite[Section 16.1]{CL})
$$K_0(\RDef_k)[\mathbb L^{-1}]\to K_0(\RDef_k(\mathcal L_{\DP,\P}(k),T_{\acl}))[\mathbb L^{-1}] \cong \mathscr M_k.$$

\begin{theorem}\label{comparison}
Let $\mathsf X$ be an invariant bounded definable subassignment of $h[m,n,0]$ such that, for every $(x,y)$ on $\mathsf X$ with $x=(x_1,\dots,x_m)$, $\ord_t x_i\geq 0$ for all $1\leq i\leq m$. Then the identity 
$$\mu(\mathsf X)=\loc(\vol(\mathsf X))$$ 
holds in $\mathscr M_{\loc}$. 
\end{theorem}

\begin{proof}
Assume that $\mathsf X$ is $\beta$-invariant, for some $\beta$ in $\mathbb N^*$. For $k(\!(t)\!)$-coordinates $(x_1,\dots,x_m)$ in $\mathsf X$, let us write
$$x_i=a_{i0}+a_{i1}t+\cdots+a_{i,\beta-1}t^{\beta-1}+\cdots,\quad 1\leq i\leq m.$$
Consider the inclusion
$$\mathsf i: h[m,n,0] \hookrightarrow h[m,\beta m+n,0]$$
given by $(x,y)\mapsto (x,(a_{i0},a_{i1},\dots,a_{i,\beta-1})_{1\leq i\leq m},y)$, and the projection 
$$\mathsf{pr}: h[m,\beta m+n,0] \to h[0,\beta m+n,0]$$
given by $(x,z)\mapsto z$. Denote by $\mathsf i_{\mathsf X}$ the restriction of $\mathsf{i}$ on $\mathsf{X}$. We can regard $\mathsf i_{\mathsf X}$ as an inclusion 
$$\mathsf i_{\mathsf X}: \mathsf X \to \mathsf X[0,\beta m,0].$$ 
Denote by $\mathsf{pr}_{\mathsf X}$ the restriction of $\mathsf{pr}$ on $\mathsf X[0,\beta m,0]$. Then the composition $\mathsf{pr}_{\mathsf X}\circ \mathsf i_{\mathsf X}$ is nothing but the canonical map $\pi_{\beta}: \mathsf X\to h_{X_{\beta}}$. By the functoriality (Axiom A0) of the integral in \cite[Theorem 10.1.1]{CL} we have $(\pi_{\beta})_!=(\mathsf{pr}_{\mathsf X})_!\circ (\mathsf i_{\mathsf X})_!$, hence 
$$(\pi_{\beta})_!\big(\big[\mathbf 1_{\mathsf X}\big]\big)=(\mathsf{pr}_{\mathsf X})_!\big(\big[\mathbf 1_{\mathsf i(\mathsf X)}\big]\big).$$
Since $\mathsf X$ is $\beta$-invariant, by fixing an element $(a,b)$ in $\mathsf i(\mathsf X)$ we have 
$$\mathsf i(\mathsf X)=\mathsf{pr}_{\mathsf X}^{-1}(h_{X_{\beta}})=\left\{(a,b)+(x,y) \in h[m,\beta m+n,0] \mid \ord_tx_i\geq \beta \ \forall \ 1\leq i\leq m \right\}.$$ 
By definition, constructible motivic Functions on $\mathsf i(\mathsf X)$ are equivalence classes of elements of $C_+(\mathsf i(\mathsf X))$ modulo support of smaller dimension (cf. \cite[Section 3.3]{CL2005}, \cite[Section 6]{CL}), hence in $\I C_+(\mathsf i(\mathsf X))$ we have $\big[\mathbf 1_{\mathsf i(\mathsf X)}\big]=\big[\mathbf 1_{\widetilde{\mathsf X}}\big]$, where $\widetilde{\mathsf X}$ is defined similarly as $\mathsf i(\mathsf X)$ with $\ord_tx_i=\beta$ replacing $\ord_tx_i\geq \beta$. Now, applying Axiom A7 in \cite[Theorem 10.1.1]{CL} inductively, we get 
$$(\mathsf{pr}_{\mathsf X})_!\big(\big[\mathbf 1_{\mathsf i(\mathsf X)}\big]\big)=\mathbb L^{-(\beta+1)m}\big[\mathbf 1_{h_{X_{\beta}}}\big].$$ 

The projection $\mathsf f$ of $\mathsf X$ onto the final object $h_{\Spec k}$ in $\Def_k$ can be factorized through the canonical map $\pi_{\beta}: \mathsf X\to h_{X_{\beta}}$, namely, we have the following commutative diagram
\begin{displaymath}
\xymatrixcolsep{5pc}\xymatrix{
	\mathsf X \ar[r]^{\pi_{\beta}} \ar[rd]_{\mathsf f}
	&h_{X_{\beta}}\ar[d]^{\widetilde{\mathsf f}}\\
	&h_{\Spec k}.
}
\end{displaymath}
Therefore, 
$$\mu(\mathsf X)=\mathsf f_!([\mathbf 1_{\mathsf X}])=\widetilde{\mathsf f}\big(\mathbb L^{-(\beta+1)m}\big[\mathbf 1_{h_{X_{\beta}}}\big]\big).$$
%Since $h_{X_{\beta}}\subseteq h[0,\beta m +n,0]$, it follows from Axiom A5 in \cite[Theorem 10.1.1]{CL} that 
By \cite[Proposition 5.3.1]{CL}, we have
\begin{align*}
\mathscr C_+(h[0,\beta m,0])&\cong SK_0(\RDef_{h[0,\beta m,0]})\otimes_{\mathscr P_+^0(h_{\Spec k})}\mathscr P_+(h_{\Spec k})\\
&\cong SK_0(\RDef_{h[0,\beta m,0]})\otimes_{\mathbb N[\L-1]}\mathbb A_+.
\end{align*}
Thus, the element $\big[\mathbf 1_{h_{X_{\beta}}}\big]$ in $\mathscr C_+(h_{X_{\beta}})$ can be written as $\big[h_{X_{\beta}}\to h_{X_{\beta}}\big]\otimes 1$. Then we have the identities
$$\mu(\mathsf X)=\mathbb L^{-(\beta+1)m}\big[h_{X_{\beta}}\to h_{X_{\beta}}\to h_{\Spec k}\big]=\mathbb L^{-(\beta+1)m}\big[h_{X_{\beta}}\big]=\loc(\vol(\mathsf X))$$
hold true in $\mathscr M_{\loc}$. 
\end{proof}

Let us recall some settings in Section 14.5 of \cite{CL} and Section 4.3 of \cite{LN2020} on the ramification. Consider a formula $\varphi$ in the language $\mathcal L_{\DP,\P}(k[t])$, i.e., the coefficients of $\varphi$ are in $k[t]$ in the valued field sort, in $k$ in the residue field sort, such that $\varphi$ has $m$ free variables in the valued field sort, $n$ free variables in residue field sort, and $r$ free variables in the value group sort. For each $e$ in $\mathbb N^*$, let $\varphi[e]$ denote the formula obtained from $\varphi$ by replacing everywhere $t$ by $t^e$. If $\mathsf X$ is a definable subassignment of $h[m,n,r]$ defined by $\varphi$, we denote by $\mathsf X[e]$ the definable subassignment of $h[m,n,r]$ defined by $\varphi[e]$. Thus, if in addition $\mathsf X$ is bounded, then so is $\mathsf X[e]$, that is, $[\mathbf 1_{\mathsf X[e]}]$ is in $\I C_+(\mathsf X[e])$, for every $e$ in $\mathbb N^*$ (cf. \cite[Proposition 14.5.1]{CL}).

\begin{proposition}[\cite{CL}, Theorem 14.5.3]\label{prop55}
Assume $\mathsf X$ is a bounded definable subassignment of $h[m,0,0]$ defined by a formula in $\mathcal L_{\DP,\P}(k[t])$ with $m$ free variables in the valued field sort. Then the formal power series 
$$Z_{\mathsf X}(T)=\sum_{e\in \mathbb N^*}\mu(\mathsf X[e])T^e$$
is in $\mathscr M_{\loc}[[T]]_{\sr}$.
\end{proposition}

Let $\beta$ be in $\mathbb N^*$, and $\mathsf X$ a $\beta$-invariant bounded definable subassignment of $h[m,n,0]$ defined by a formula in $\mathcal L_{\DP,\P}(k[t])$. Then, for every $e$ in $\mathbb N^*$, the bounded definable subassignment $\mathsf X[e]$ of $h[m,n,0]$ is $\beta e$-invariant; therefore, there exists a constructible subset $X_{\beta e}$ of $\mathscr L_{\beta e-1}(\mathbb A_k^m) \times_k \mathbb A_k^n$ such that, for every $K$ in $\Field_k$, $\mathsf X[e](K)$ is the pullback of $X_{\beta e}(K)$ under the canonical map 
$$K(\!(t)\!)^m \times K^n \to \big(K[t]/(t^{\beta e})\big)^m\times K^n.$$

\begin{proposition}\label{prop56}
Let $\beta$ be in $\mathbb N^*$, and let $\mathsf X$ be a $\beta$-invariant bounded definable subassignment of $h[m,n,0]$ defined by a formula $\varphi$ in $\mathcal L_{\DP,\P}(k[t])$ not containing the symbol $\overline{\ac}$. Then the definable subassignment $h_{X_{\beta e}}$ is stable by the natural $\mu_e$-action on $h[m,n,0]$ defined by 
$$\lambda\cdot (x,\xi)=(\lambda x,\xi),\quad \text{with}\quad \lambda x(t)=x(\lambda t).$$ 
As a consequence, the quantity $\vol({\mathsf X}[e])$ is an element in $\mathscr M_k^{\mu_e}$, and the identity 
$$\mu(\mathsf X[e])=\loc(\vol({\mathsf X}[e]))$$ 
holds in $\mathscr M_{\loc}^{\mu_e}$, thus the series 
$$Z_{\mathsf X}(T)=\sum_{e\in \mathbb N^*}\mu(\mathsf X[e])T^e$$
is in $\mathscr M_{\loc}^{\hat\mu}[[T]]_{\sr}$.
\end{proposition} 

\begin{proof}
By definition of $\beta$-invariance, we have
$$h_{X_{\beta e}}\cong \left\{(x,\xi)\in h[m,n,0] \mid \varphi[e](x,\xi), 0\leq \ord_tx_i \leq \beta e-1, 1\leq i\leq m\right\}.$$
By Corollary 2.1.2 of \cite{CL}, the $\mathcal L_{\DP,\P}(k[t])$-formula $\varphi$ is a finite disjunction of formulas of the form 
$$\psi(\overline{\ac}g'_1(x),\dots,\overline{\ac}g'_{l'}(x),\xi)\wedge \vartheta(\ord_t f'_1(x),\dots,\ord_t f'_l(x),\alpha),$$
where $f'_i$ and $g'_j$ are polynomials over $k[t]$, $\psi$ is an $\mathbf L_{\Rings}$-formula, and $\vartheta$ is an $\mathbf L_{\PR}$-formula. Since $\varphi$ is equivalent to an $\mathcal L_{\DP,\P}(k[t])$-formula without angular component symbol $\overline{\ac}$ due to the hypothesis, $\varphi[e]$ is a finite disjunction of formulas of the form 
$$\psi_1((g_j(x))_j)\wedge \psi_2(\xi)\wedge \vartheta((\ord_t f_i(x))_i,\alpha),$$ 
where $f_i$ and $g_j$ are polynomials over $k[t^e]$, $\psi_1$ and $\psi_2$ are $\mathbf L_{\Rings}$-formulas, and $\vartheta$ is an $\mathbf L_{\PR}$-formula. If in $f_i(x(t))$ and $g_j(x(t))$ we replace $t$ by $\lambda t$, for any $x(t)$ in $K(\!(t)\!)^m$, $\lambda$ in $\mu_e(K)$, $K$ in $\Field_k$, then we get expressions of $f_i(x(\lambda t))$ and $g_j(x(\lambda t))$, since the coefficients of polynomials $f_i$ and $g_j$ in $k[t^e][x]$ do not change. This proves that if $(x,\xi)$ is in $h_{X_{\beta e}}$, so is $(\lambda x,\xi)$, for any $\lambda$ in $\mu_e$, that is, $h_{X_{\beta e}}$ is stable under the action of $\mu_e$.
\end{proof}

Now assume that $\mathsf X$ is small. As mentioned above, there is a canonical action of $\hat\mu$ on $h[m,0,0]$ induced by 
$$(\lambda,t)\mapsto \lambda t.$$ 
We say that the definable subassignment $\mathsf X$ is $\hat\mu$-stable if there exists an $n\in \mathbb N^*$ such that, for every $x=(x_1(t),\dots,x_m(t))$ in $\mathsf X$ and $\lambda$ in $\mu_n$, the point 
$$\lambda \cdot x=(x_1(\lambda t),\dots,x_m(\lambda t))$$ 
is in $\mathsf X$. Since formulas defining $\mathsf X$ is in the Denef-Pas language, by quantifier elimination for algebraically closed fields, they also define a semi-algebraic subset $X$ of some arc space $\mathscr L(\mathbb A_k^m)$ of $\mathbb A_k^m$. The assignment 
$$\mathsf X \mapsto X$$ 
carries the canonical $\hat\mu$-action on $h[m,0,0]$ to the canonical $\hat\mu$-action on $\mathscr L(\mathbb A_k^m)$, and in that way, $X$ is also $\hat\mu$-stable in $\mathscr L(\mathbb A_k^m)$. As in \cite[Theorem 16.3.1, Remark 16.3.2]{CL} we can see that $X$ is measurable as $\mathsf X$ is measurable, and that since with the above action $\mu'(X)$ is in $\mathscr M_{\loc}^{\hat\mu}$ the measure $\mu(\mathsf X)$ of $\mathsf X$ is also in $\mathscr M_{\loc}^{\hat\mu}$. Here, as in \cite[Theorem 16.3.1]{CL}, $\mu'$ stands for Denef-Loeser's motivic measure \cite{DL2}, and further by \cite[Remark 16.3.2]{CL}, we can consider that this measure takes value in $\mathscr M_{\loc}^{\hat\mu}$ for the context with $\hat\mu$-action.

%********************************

\begin{ack}
Part of the present paper was worked out at the Vietnam Institute for Advanced Study in Mathematics (VIASM), the author would like to thank the institute for warm hospitality. 
\end{ack}

\end{document}